\documentclass[11pt,a4paper]{article}

\title{Necessity of weak subordination for some strongly subordinated L\'evy processes}
\author{Boris Buchmann\thanks{Research School of Finance, Actuarial Studies \& Statistics,
        Australian National University,
        ACT 0200,
        Australia.
        Email: \href{mailto:boris.buchmann@anu.edu.au}{boris.buchmann@anu.edu.au}
    }
    \and
    Kevin W.~Lu\thanks{Research School of Finance, Actuarial Studies \& Statistics,
        Australian National University,
        ACT 0200,
        Australia.
        Email: \href{mailto:u5119413@anu.edu.au}{u5119413@anu.edu.au}}
}

\usepackage{amsmath,amsfonts,amssymb,amsthm}
\usepackage{graphicx}

\usepackage[shortlabels]{enumitem} 
\usepackage[colorlinks,citecolor=red,linkcolor=blue]{hyperref} 
\usepackage{fancyhdr} 
\usepackage{microtype} 
\usepackage{datetime2} 

\numberwithin{equation}{section} 
\allowdisplaybreaks[4] 
\let\originalleft\left
\let\originalright\right
\def\left#1{\mathopen{}\originalleft#1}
\def\right#1{\originalright#1\mathclose{}}

\newtheorem{theorem}{Theorem}[section]

\newtheorem{lemma}[theorem]{Lemma}
\newtheorem{proposition}[theorem]{Proposition}

\theoremstyle{definition}
\newtheorem{definition}[theorem]{Definition}

\theoremstyle{remark}
\newtheorem{remark}[theorem]{Remark}

\renewcommand{\today}
{\number\day \space \ifcase\month\or
    January\or February\or March\or April\or May\or June\or
    July\or August\or September\or October\or November\or December
    \fi \space \number\year
}

\newcommand{\rmd}{{\rm d}}
\newcommand{\rmi}{{\rm i}}

\newcommand{\eqd}{\stackrel{D}{=}}

\newcommand{\CC}{\mathbb{C}}
\newcommand{\DD}{\mathbb{D}}
\newcommand{\EE}{\mathbb{E}}

\newcommand{\RR}{\mathbb{R}}
\newcommand{\NN}{\mathbb{N}}

\newcommand{\ZZ}{\mathbb{Z}}

\newcommand{\PP}{\mathbb{P}}

\newcommand{\BBB}{{\cal B}}

\newcommand{\EEE}{{\cal E}}

\newcommand{\RRR}{{\cal R}}

\newcommand{\TTT}{{\cal T}}

\newcommand{\XXX}{{\cal X}}

\newcommand{\ZZZ}{{\cal Z}}

\newcommand{\skal}[2]{\langle #1,#2\rangle}

\newcommand{\eins}{{\bf 1}}

\newcommand{\bfnull}{{\bf 0}}

\newcommand{\bfe}{{\bf e}}
\newcommand{\bfm}{{\bf m}}
\newcommand{\bfr}{{\bf r}}

\newcommand{\bft}{{\bf t}}
\newcommand{\bfx}{{\bf x}}
\newcommand{\bfy}{{\bf y}}
\newcommand{\bfw}{{\bf w}}
\newcommand{\bfz}{{\bf z}}
\newcommand{\bfd}{{\bf d}}

\newcommand{\bfR}{{\bf R}}
\newcommand{\bfS}{{\bf S}}
\newcommand{\bfT}{{\bf T}}

\newcommand{\bfX}{{\bf X}}
\newcommand{\bfY}{{\bf Y}}
\newcommand{\bfZ}{{\bf Z}}

\newcommand{\bfdelta}{\boldsymbol{\delta}}

\newcommand{\bflambda}{\boldsymbol{\lambda}}
\newcommand{\bfmu}{\boldsymbol{\mu}}
\newcommand{\bfpi}{\boldsymbol{\pi}}

\newcommand{\bftheta}{\boldsymbol{\theta}}

\newcommand{\tr}{\diamond}

\newcommand{\given}{{\,\vert\,}}
\newcommand{\givenm}{\,\middle\vert\,}

\begin{document}
\date{\today}
\maketitle

\begin{abstract}
    Consider the strong subordination of a multivariate L\'evy process with a multivariate subordinator. If the subordinate is a stack of independent L\'evy processes and the components of the subordinator are indistinguishable within each stack, then strong subordination produces a L\'evy process, otherwise it may not. Weak subordination was introduced to extend strong subordination, always producing a L\'evy process even when strong subordination does not.  Here, we prove that strong and weak subordination are equal in law under the aforementioned condition. In addition, we prove that if strong subordination is a L\'evy process, then it is necessarily equal in law to weak subordination in two cases: firstly, when the subordinator is deterministic and secondly, when it is pure-jump with finite activity.
    
    \vspace{0.5em} \noindent {\em Keywords:} L\'evy process, subordinator, multivariate subordination, weak subordination, Poisson point process, Poisson random measure.
       
    \vspace{0.5em} \noindent {\em 2010 MSC Subject Classification:} Primary: 60G51, 60G55.
\end{abstract}

\section{Introduction}

Let $\bfX=(X_1,\dots, X_n)$ be an $n$-dimensional L\'evy process and $\bfT=(T_1,\dots,T_n)$ be an $n$-dimensional subordinator independent of $\bfX$. The operation that evaluates the process $\bfX$ at times given by the subordinator $\bfT$ is defined by
\begin{align*}
    \bfX\circ\bfT = (X_1(T_1(t)),\dots,X_n(T_n(t)))_{t\geq 0}
\end{align*}
and known as \emph{strong subordination}. This creates a ``time-changed" process. The study of the multivariate subordination of L\'evy processes originated with the work of \cite{BPS01}. It is well-known that strong subordination produces a L\'evy process in the following cases:
\begin{enumerate}[(C1)]
    \item $\bfT$ has indistinguishable components. \label{case1}
    \item $\bfX$ has independent components. \label{case2}
    \item $\bfT$ and $\bfX$ satisfy the \emph{stacked univariate subordination} condition: for some $1\leq d\leq n$ and $n_1+\dots+n_d=n$,
    \begin{align}\label{stacksubord}
        \bfX=(\bfY_1,\dots,\bfY_d),\quad \bfT = (R_1\bfe_1,\dots ,R_d\bfe_d),
    \end{align}
    where $\bfY_1,\dots,\bfY_d$ are independent L\'evy processes, $\bfY_m$ is $n_m$-dimensional, $(R_1,\allowbreak\dots, R_d)$ is a $d$-dimensional subordinator and $\bfe_m=(1,\dots,1)\in\RR^{n_m}$, $1\leq m \leq d$.  \label{case3}
\end{enumerate}
For the proof of sufficiency under conditions \ref{case1} and \ref{case3}, see \cite[Theorem 30.1]{sato99} and \cite[Theorem 3.3]{BPS01}, respectively, though the  origin of the former case goes back to \cite{Zo58}. Both conditions \ref{case1} and \ref{case2} are implied by condition \ref{case3}. It is condition \ref{case2}, as opposed to condition \ref{case3}, that more commonly appears in financial applications \cite{BKMS16,LS10,Se08} since it has a more intuitive interpretation. Outside of these sufficient conditions, strong subordination does not necessarily produce a L\'evy process \cite[Proposition 3.9]{BLM17a}.

These restrictions on $\bfX$ and $\bfT$ for $\bfX\circ\bfT$ to remain in the well-understood L\'evy process framework are  problematic in applications because they severely limit the dependence structure of $\bfX\circ\bfT$. To address this shortcoming, \cite{BLM17a} introduced a new operation for constructing general time-changed multivariate L\'evy processes $\bfX\odot\bfT$, known as \emph{weak subordination}, without any restriction on the subordinate $\bfX$ or the subordinator~$\bfT$.

Weak subordination is based on the idea, roughly speaking, of constructing the L\'evy process that has the distribution of $\bfX(\bft):=(X_1(t_1),\dots,X_n(t_n))$ conditional on $\bfT(t) = \bft:=(t_1,\dots,t_n)$, $t\geq0$. To be more detailed, the idea is to decompose the subordinator $\bfT(t) = \bfd t + \bfS(t)$, $t\geq0$,  into deterministic and pure-jump parts. For the deterministic subordinator part, $\bfX(\bfd)$ is infinitely divisible and hence associated to a L\'evy process, while for the pure-jump subordinator part, a marked Poisson point process can be constructed such that it jumps with the distribution of $\bfX(\bft)$ when the subordinator jumps by $\Delta \bfT(t) = \bft$ and then associated to a L\'evy process by the L\'evy-It\^{o} decomposition, and finally the two L\'evy processes are combined by convolution. This allows for more flexible dependence modelling while remaining in the class of L\'evy processes, a closure property not enjoyed by strong subordination. Weak subordination coincides with strong subordination under conditions \ref{case1} or \ref{case2} in the sense that $(\bfT,\allowbreak\bfX\circ \bfT)\eqd (\bfT,\bfX\odot \bfT)$ \cite[Proposition 3.3]{BLM17a}, and it also reproduces many analogous properties \cite[Propositions 3.3, 3.7]{BLM17a}.

In this paper, we show that the more general case \ref{case3} of stacked univariate subordination considered in \cite{BPS01} also satisfies $(\bfT,\bfX\circ \bfT)\eqd (\bfT,\bfX\odot \bfT)$, which unifies it under weak subordination (see Theorem \ref{propsupextendssub}). This raises the question of whether there are alternative definitions of weak subordination that are also consistent with strong subordination in this way. We partially address this by showing that if $(\bfT, \bfX\circ \bfT)$ is a L\'evy process, then $(\bfT, \bfX\circ \bfT)\eqd(\bfT, \bfX\odot \bfT)$ in two cases: $\bfT$ is deterministic (see Proposition~\ref{detsubordcase}), or $\bfT$ is a pure-jump subordinator with finite activity (see Theorem \ref{orderjumpcaseprop}). In the former case, we can weaken the assumption to $\bfX\circ \bfT$ being a L\'evy process. Our proof tracks the construction of weak subordination in the deterministic subordinator and pure-jump subordinator cases mentioned above, and in the latter, the theory of marked Poisson point processes is used to verify that the relevant characteristics coincide.

We briefly mention some applications. The subordination of L\'evy processes is used in mathematical finance to create time-changed models of stock prices. This idea began with the work of \cite{MaSe90} who introduced the variance gamma (VG) process for modelling stock prices, created by subordinating a Brownian motion with a gamma subordinator. Subordination can also be applied to model dependence in multivariate price processes. The multivariate VG process in~\cite{MaSe90} was created by subordinating multivariate Brownian motion with a univariate gamma subordinator, so the components cannot have idiosyncratic time changes and must have equal kurtosis when there is no skewness. These deficiencies were addressed by the use of an alpha-gamma subordinator, resulting in the variance alpha-gamma process which was introduced in~\cite{Se08} and also considered in~\cite{Gu13,LS10}. However, in this case, the Brownian motion subordinate must have independent components. In both models, the use of strong subordination to create a L\'evy process restricts the dependence structure. By using weak subordination instead, and a general Brownian motion, \cite{BLM17c,BLM17a} introduced the weak variance alpha-gamma (WVAG) process to provide additional flexibility in dependence modelling while remaining tractable. The WVAG process exhibits a wider range of dependence while remaining parsimoniously parametrised, the marginal components have both common and idiosyncratic time changes, and are VG processes with possibly different levels of kurtosis. These weakly subordinated processes have been applied to option pricing \cite{MiSz17} and instantaneous portfolio theory \cite{Ma17}.

The paper is structured as follows. In Section \ref{sect2}, we review some notations, definitions, and preliminary results relating to L\'evy processes, weak subordination and Poisson random measures. In Section  \ref{sect3}, we state and prove the main result, namely that weak subordination is consistent with strong subordination under condition \ref{case3}, and that if strong subordination produces a L\'evy process it is necessarily equal in law to weak subordination when the subordinator is deterministic or pure-jump with finite activity. We conclude in Section \ref{sect4} with a brief discussion placing this work in the context of open questions relating to the subordination of multivariate L\'evy processes.

\section{Preliminaries}\label{sect2}

\subsection{L\'evy processes}
We write $\bfx=(x_1,\dots,x_n)\in\RR^n$ as a row vector. For $A\subseteq\RR^n$, let $A_*:=A\backslash\{{\bf 0}\}$ and let $\eins_A$ denote the indicator function for $A$. Let $\DD:=\{\bfx\in\RR^n:\|\bfx\|\le 1\}$ be the Euclidean unit ball centred at the origin. Let $\|\bfx\|^2_\Sigma:=\bfx\Sigma\bfx'$, where $\bfx\in\RR^n$ and $\Sigma\in\RR^{n\times n}$. Let $I:[0,\infty)\to[0,\infty)$ be the identity function.

For references on L\'evy processes, see \cite{bert96,sato99}. The law of an $n$-dimensional L\'evy process $\bfX=(X_1,\dots,X_n)=(\bfX(t))_{t\ge 0}$ is determined by its characteristic function $\Phi_\bfX:=\Phi_{\bfX(1)}$ with
\begin{align*}
    \Phi_{\bfX(t)}(\bftheta)\,:=\,\EE\exp(\rmi\skal\bftheta{\bfX(t)}) = \exp(t\Psi_\bfX(\bftheta)), \quad \text{$t\ge0$, $\bftheta\in\RR^n$},
\end{align*}
and characteristic exponent
\begin{align*}
    \Psi_\bfX(\bftheta):=
    \rmi \skal {\bfmu}\bftheta - \frac 12\|\bftheta\|^2_{\Sigma} +\int_{\RR^n_*}(e^{\rmi\skal\bftheta \bfx} - 1 - \rmi\skal\bftheta \bfx \eins_\DD(\bfx))\,\XXX(\rmd \bfx),
\end{align*}
where $\bfmu\in\RR^n$, $\Sigma\in\RR^{n\times n}$ is a covariance matrix,
and $\XXX$ is a L\'evy measure, that is a nonnegative Borel measure on $\RR^n_*$ such that
$\int_{\RR^n_*}(1\wedge \|\bfx\|^2)\,\XXX(\rmd \bfx)<\infty$. We write $\bfX\sim L^n(\bfmu,\Sigma,\XXX)$ (or $\bfX\sim L^n$ for short) to mean $\bfX$ is an $n$-dimensional L\'evy process with characteristic triplet $(\bfmu,\Sigma,\XXX)$.

An $n$-dimensional L\'evy process $\bfT$ with almost surely nondecreasing sample paths is called a subordinator and it is characterised by $\bfT \sim S^n(\bfd,\TTT) := L^n(\bfmu,0,\TTT)$ (or $\bfT\sim S^n$ for short), where $\bfd := \bfmu - \int_{\DD_*} \bft \,\TTT(\rmd \bft)\in[0,\infty)^n$ is the drift, and the L\'evy measure $\TTT$ satisfies $\TTT(([0,\infty)^n)^C) =0 $. The law of $\bfT$ is also characterised by its Laplace exponent $\Lambda_\bfT$, which satisfies $\EE[\exp(-\skal{\bflambda}{\bfT(1)})]=\exp(-\Lambda_\bfT(\bflambda))$, $\bflambda\in[0,\infty)^n$. For $\bfw,\bfz \in\CC^n$, let $\skal{\bfw}{\bfz} :=\sum_{k=1}^n w_kz_k$, noting that there is no conjugation. The domain of $\Lambda_\bfT$ can be extended, giving
\begin{align}\label{subordlaplaceexp}
    \Lambda_\bfT(\bfz)=\skal{\bfd}{\bfz}+\int_{[0,\infty)_*^n}(1-e^{-\skal\bfz\bft}) \,\TTT(\rmd \bft),\quad\Re\bfz\in[0,\infty)^n
\end{align}
(see the proof of \cite[Theorem 3.3]{BPS01}).

We have the following characterisation of piecewise constant L\'evy processes from \cite[Theorem 21.2]{sato99}.
\begin{lemma}\label{lem2}
    Let $\bfX\sim L^n(\bfmu,\Sigma,\XXX)$. The following are equivalent:
    \begin{enumerate}[(i)]
        \item $\bfX$ is piecewise constant a.s.;
        \item $\bfX$ is driftless, $\Sigma=0$ and $\XXX(\RR^n_*)<\infty$;
        \item $\bfX$ is a compound Poisson process or $\bfX$ is the zero process.
    \end{enumerate}
\end{lemma}

\subsection{Weak subordination}
Let $\bfX\sim L^n$ and $\bfT\sim S^n(\bfd,\TTT)$. Let $\bft=(t_1,\dots,t_n)\in[0,\infty)^n$ and $\langle{(1),\dots,(n)}\rangle$ be a permutation of $\{1,\dots,n\}$ such that $t_{(1)} \leq \dots \leq t_{(n)}$, and define $\Delta t_{(k)}:=t_{(k)}-t_{(k-1)}$, $1\leq k \leq n$, with $t_{(0)}:=0$. Further, let  $\bfpi_{J}:\RR^n\to\RR^n$ be the projection onto the coordinate axes in $J\subseteq\{1,\dots,n\}$. For all $\bft\in[0,\infty)^n$, by \cite[Proposition 2.1]{BLM17a}, the random vector $\bfX(\bft):= (X_1(t_1),\dots, X_n(t_n))$ is infinitely divisible with characteristic exponent
\begin{align}
    (\bft\tr \Psi_\bfX)(\bftheta):={}&\sum_{k=1}^n\Delta t_{(k)} \Psi_\bfX(\bfpi_{\{(k),\dots,(n)\}}(\bftheta)),\quad \bftheta\in\RR^n, \label{multiexponent1}
\end{align}
and characteristic function
\begin{align}\label{xvecteqn}
    \Phi_{\bfX(\bft)}(\bftheta) = \exp({\bft \tr \Psi_{\bfX}(\bftheta)}).
\end{align}

It is convenient to consider both the subordinator and the subordinated process together as a joint $2n$-dimensional L\'evy process. In this form, we can define weak subordination as follows (see \cite[Proposition 3.1]{BLM17a}). 
\begin{definition}
    The \emph{weak subordination} of $\bfX\sim L^n$ and $\bfT\sim S^n(\bfd,\TTT)$ is the joint L\'evy process $\bfZ\eqd(\bfT,\bfX\odot \bfT)$ with characteristic exponent
    \begin{align}\label{propsupervisexpo}
        \Psi_\bfZ(\bftheta) = \rmi\skal{\bfd}{\bftheta_1}+(\bfd\tr\Psi_{\bfX})(\bftheta_2)+\int_{[0,\infty)^n_*} (\Phi_{(\bft,\bfX(\bft))}(\bftheta)-1)\,\TTT(\rmd \bft),
    \end{align}
    where $\bftheta=(\bftheta_1,\bftheta_2)$, $\bftheta_1,\bftheta_2\in\RR^n$.
\end{definition}
This is a valid characteristic exponent for a L\'evy process. Weak subordination can equivalently be defined in terms of a characteristic triplet. From \cite[Definition 2.1]{BLM17a}, if $\bfd =\bfnull$, then $\bfZ\sim L^{2n}(\bfm,\Theta,\ZZZ)$, where
\begin{align}
    \bfm &{}= \int_{\DD_*} (\bft,\bfx)\,\ZZZ(\rmd\bft,\rmd\bfx),\label{chartri1}\\
    \Theta &{}=0,\label{chartri2}\\
    \ZZZ(\rmd\bft,\rmd\bfx) &{} = \eins_{[0,\infty)^n_*\times \RR^n}(\bft,\bfx) \PP(\bfX(\bft)\in\rmd\bfx)\TTT(\rmd\bft).\nonumber
\end{align}
For additional details on weak subordination, see \cite{BLM17a}.

%
%
\subsection{Poisson random measures}
For references on Poisson random measures and their relationship to the jumps of L\'evy processes, see \cite{bert96,Cin11,Kin93}. A Poisson random measure (PRM) $\ZZ$ with intensity measure $\mu$ on a measurable space $(E,\EEE)$ is a random measure such that $\ZZ(A)\sim \operatorname{Poisson}(\mu(A))$ for all $A\in\EEE$, and $\ZZ(A_1),\dots\ZZ(A_m)$ are independent for all disjoint $A_1,\dots,A_m\in\EEE$. In general, a PRM has the form $\ZZ=\sum_{i=1}^\infty \bfdelta_{\bfZ_i}$, where $\bfdelta_{\bfZ_i}$, $i\in\NN$, is the Dirac measure at the random vector $\bfZ_i$ taking values in $(E,\EEE)$. Define the random variable $Z_f:=\int_{E} f(\bfx)\,\ZZ(\rmd \bfx)$, where $f$ is a nonnegative, $\EEE$-measurable real function.  The Laplace functional of $\ZZ$ is
\begin{align}\label{genlf}
    L(f):=\EE\left[ e^{-Z_f}\right] = \EE\left[\prod_{i=1}^\infty e^{-f(\bfZ_i)} \right] = \exp\left(-\int_E (1-e^{-f(\bfx)})\,\mu(\rmd \bfx)  \right) 
\end{align}
(see \cite[Equation (3.35)]{Kin93}). The Laplace functional is well-defined with $L(f)\in[0,1]$, where this equality can be interpreted as 0 if $Z_f<\infty$ a.s.\ fails. 
Two PRMs are equal if their Laplace functionals are \cite[Chapter VI, Proposition 1.4]{Cin11}.


Let $\bfX\sim L^n(\bfmu,\Sigma,\XXX)$ with $\XXX\neq 0$. For a fixed sample path, a time $t$ is a jumping time of $\bfX$ if the jump $\Delta \bfX(t):=\bfX(t)-\bfX(t-)\neq \bfnull$. The following result is \cite[Theorem~21.3]{sato99}.

\begin{lemma}\label{countjumplem}
    Let $\bfX\sim L^n(\bfmu,\Sigma,\XXX)$ with $\XXX\neq 0$, then its jumping times are countably infinite. Denoting these jumping times as $\bfS= (S_i)_{i\in\NN}$, we have in addition:
    \begin{enumerate}[(i)]
        \item if $\XXX(\RR^n_*)<\infty$, then $\bfS$ is countable in increasing order, or
        \item if $\XXX(\RR^n_*)=\infty$, then $\bfS$ is dense in $[0,\infty)$.
    \end{enumerate}
\end{lemma}
For the L\'evy process $\bfX$, the countable sequence of random vectors giving the time and size of the jumps, $(\bfZ_{i})_{i\in\NN} := (t,\Delta \bfX(t))_{t>0,\Delta \bfX(t)\neq \bfnull}$, is a Poisson point process. Consequently, $\ZZ=\sum_{i=1}^\infty \bfdelta_{\bfZ_i}$ is the PRM of $\bfX$ (or of the jumps of $\bfX$), defined on the Borel space $([0,\infty)\times\RR^n_*,\BBB([0,\infty)\times\RR^n_*))$ with intensity measure $\rmd t \otimes \XXX$ \cite[Chapter I, Theorem 1]{bert96}.

%


\section{Main results}\label{sect3}

\subsection{Consistency of weak subordination for stacked univariate subordination}

Here, we show that the law of weak and strong subordination coincide when the latter satisfies the stacked univariate subordination property in condition \ref{case3}, and hence is a L\'evy process. The proof follows along the lines of \cite[Proposition 3.3]{BLM17a}.

\begin{theorem}\label{propsupextendssub} Let $\bfT\sim S^n$ and $\bfX\sim L^n$ be independent. If $\bfT$ and $\bfX$ satisfy the stacked univariate subordination condition in \eqref{stacksubord}, then $(\bfT,\bfX\circ \bfT)\eqd(\bfT,\bfX\odot \bfT)$.
    
\end{theorem}

\begin{proof}
    Recall that $n_1+\dots+n_d=n$ and let $\bftheta=(\bftheta_1,\bftheta_2) = (\bftheta_{11},\dots \bftheta_{1d},\bftheta_{21},\dots,\allowbreak \bftheta_{2d})$, $\bftheta_1,\bftheta_2\in\RR^n$, $\bftheta_{1m},\bftheta_{2m}\in\RR^{n_m}$ for all $1\leq m \leq d$. Since $\bfT$ and $\bfX$ are independent processes, using \eqref{xvecteqn} and conditioning on $\bfT$, we get
    \begin{align}\label{charsubord}
        \Phi_{(\bfT,\bfX\circ \bfT)}(\bftheta)=\EE[\exp(\rmi \skal{\bftheta_1}{\bfT(1)}+(\bfT(1)\tr\Psi_{\bfX})(\bftheta_2))].
    \end{align}
    Let $\bfe=(1,\dots,1)\in\RR^n$. Since $\bfT$ and $\bfX$ satisfy the stacked univariate subordination condition for $n$-dimensional processes, the subordinator $(\bfT,\bfT)$ and the subordinate $(I\bfe,\bfX)$ satisfy this condition for $2n$-dimensional processes. Thus, $(I\bfe,\bfX)\circ (\bfT,\bfT)= (\bfT,\bfX\circ \bfT)$ is a L\'evy process by \cite[Theorem 3.3]{BPS01}, so it suffices to show that $\Psi_{(\bfT,\bfX\circ \bfT)}=\Psi_{(\bfT,\bfX\odot \bfT)}$.

    Noting that $\bfX = (\bfY_1,\dots, \bfY_d)$, where $\bfY_1\sim L^{n_1},\dots, \bfY_d\sim L^{n_d}$ are independent L\'evy processes, Kac's theorem gives
    \begin{align*}
        \Psi_\bfX(\bftheta_2) =\sum_{m=1}^d \Psi_{\bfY_m}(\bftheta_{2m}).
    \end{align*}
    Form the partition $\{1,\dots,n\}= J_1\cup\dots \cup J_d$, where $J_1 := \{1,\dots,n_1\}, J_2 := \{n_1+1,\dots,n_1+n_2\},\dots, J_d :=\{n_1+\dots+n_{d-1}+1, \dots, n\}$. Let $\bfr=(r_1,\dots,r_d)\in[0,\infty)^d$ and $\langle(1),\dots,(d)\rangle$ be a permutation of $\{1,\dots, d\}$ such that $r_{(1)}\leq \dots \leq r_{(d)}$. Define the projections $\bfpi_m:= \bfpi_{J_{(m)} \cup\dots \cup J_{(d)}}$, $1\leq m \leq d$. Thus, for all $1\leq m\leq d$,
    \begin{align}
        \Psi_\bfX(\bfpi_m(\bftheta_2)) =\sum_{k=m}^d \Psi_{\bfY_{(k)}}(\bftheta_{2(k)}).\label{rapsi2}
    \end{align}
    Next, due to \eqref{stacksubord}, we can write $\bfT= \bfR A$ for some $A\in\RR^{d\times n}$, where $\bfR= (R_1,\dots,R_d)\sim S^d(\bfd,\RRR)$. Then \eqref{multiexponent1} and \eqref{rapsi2} gives
    \begin{align}
        (\bfr A)\tr \Psi_{\bfX}(\bftheta_2) &{}= \sum_{m=1}^d (r_{(m)}-r_{(m-1)})\Psi_{\bfX}(\bfpi_{m}(\bftheta_2))\nonumber\\
        &{}= \sum_{m=1}^{d-1} r_{(m)}(\Psi_{\bfX}(\bfpi_{m}(\bftheta_2)) - \Psi_{\bfX}(\bfpi_{m+1}(\bftheta_2))) + r_{(d)}\Psi_{\bfX}(\bfpi_{d}(\bftheta_2))\nonumber\\
        &{}= \sum_{m=1}^d r_{(m)}\Psi_{\bfY_{(m)}}(\bftheta_{2(m)})\nonumber\\
        &{}=\sum_{m=1}^d r_{m}\Psi_{\bfY_m}(\bftheta_{2m}).\label{rapsi3}
    \end{align}
    
    Let $\bfz:=(z_1,\dots,z_d)\in\CC^d$, where $z_m = -\rmi\skal{\bftheta_{1m}}{\bfe_m}-\Psi_{\bfY_m}(\bftheta_{2m})$, $1 \leq m \leq d$. Using \eqref{rapsi3}, we have
    \begin{align}
        -\skal{\bfz}{\bfr}&{}=\sum_{m=1}^d  \rmi\skal{\bftheta_{1m}}{r_m\bfe_m}+r_m\Psi_{\bfY_m}(\bftheta_{2m})\nonumber\\
        &{}=\rmi\skal{\bftheta_1}{\bfr A} + (\bfr A)\tr \Psi_{\bfX}(\bftheta_2)\label{z_prop}.
    \end{align}
    Thus, $\eqref{charsubord}$ becomes $\Phi_{(\bfT,\bfX\circ \bfT)}(\bftheta)=\EE[\exp(-\skal{\bfz}{\bfR(1)})]$. By noting that $\Re{\bf z}\in[0,\infty)^d$ and using \eqref{subordlaplaceexp} to obtain the Laplace exponent of $\bfR$, we have $\Psi_{(\bfT,\bfX\circ \bfT)}(\bftheta)=-\Lambda_\bfR(\bfz)$, where
    \begin{align*}
        \Lambda_\bfR(\bfz)=\skal{\bfd}{\bfz}+\int_{[0,\infty)_*^d}(1-e^{-\skal {\bfz}\bfr})\,\RRR(\rmd \bfr).
    \end{align*}
    Using \eqref{xvecteqn} and \eqref{z_prop}, we have $e^{-\skal {\bfz}\bfr}=\Phi_{(\bfr A,\bfX(\bfr A))}(\bftheta)$ for $\bfr\in[0,\infty)^d_*$. Thus,
    \begin{align*}
        \Psi_{(\bfT,\bfX\circ \bfT)}(\bftheta) ={}&\rmi\skal{\bftheta_1}{\bfd A} + (\bfd A)\tr \Psi_{\bfX}(\bftheta_2)+\int_{[0,\infty)^d_*} (\Phi_{(\bfr A ,\bfX(\bfr A))}(\bftheta)-1)\,\RRR(\rmd \bfr)\\
        ={}& \rmi\skal{\bftheta_1}{\bfd A} + (\bfd A)\tr \Psi_{\bfX}(\bftheta_2)\\
        {}&\quad+\int_{[0,\infty)^n_*} (\Phi_{(\bft,\bfX(\bft))}(\bftheta)-1)\,(\RRR\circ A^{-1}) (\rmd \bft)
    \end{align*}
    by the transformation theorem. This matches the RHS of \eqref{propsupervisexpo} because $\bfT \sim S^n(\bfd A, \allowbreak\RRR\circ A^{-1})$. Therefore, $(\bfT,\bfX\circ \bfT)\eqd(\bfT,\bfX\odot \bfT)$.
\end{proof}

\subsection{Necessity of weak subordination for deterministic and pure-jump, finite activity subordinators}

In this subsection, we assume that $(\bfT,\bfX\circ\bfT)\sim L^{2n}$, and under some conditions, we show that it is equal in law to $(\bfT,\bfX\odot\bfT)$. This is in contrast to Theorem~\ref{propsupextendssub}, where we proved under a condition for which it is known that $(\bfT,\bfX\circ\bfT)\sim L^{2n}$, then it is equal in law to $(\bfT,\bfX\odot\bfT)$. The conditions we consider in this subsection are that the subordinator $\bfT$ is deterministic, or that $\bfT$ is pure-jump with finite activity, which are dealt with in Proposition~\ref{detsubordcase} and Theorem \ref{orderjumpcaseprop}, respectively. In the former case, we use the weaker assumption $\bfX\circ \bfT\sim L^{n}$.

\begin{proposition}\label{detsubordcase}
    Let $\bfT\sim S^n(\bfd,0)$ and $\bfX \sim L^{n}$ be independent, with $\bfd\in[0,\infty)^n$. If  $\bfX\circ \bfT\sim L^{n}$, then $(\bfT,\bfX\circ\bfT)\eqd (\bfT,\bfX\odot\bfT)$.
\end{proposition}

\begin{proof}
    Since $\bfT$ is deterministic, $\bfT=\bfd I$. The stationary and independent increment property of the L\'evy process $\bfX\circ\bfT\sim L^{n}$ is equivalent to, for all $0\le t_1<\dots <t_{m+1}$, $m\geq1$, 
    \begin{align*}
        \hspace{2em}&\hspace{-2em}  \EE \left[ \prod_{k=1}^{m} \exp(\rmi\skal{\bftheta_{2k}}{\bfX(t_{k+1}\bfd )-\bfX(t_k\bfd)})\right] \\
        ={}&\prod_{k=1}^{m} \EE \left[ \exp(\rmi\skal{\bftheta_{2k}}{\bfX((t_{k+1}-t_k)\bfd)})\right],\quad  \bftheta_{21},\dots, \bftheta_{2m}\in\RR^n.
    \end{align*}
    Multiplying both sides by $\prod_{k=1}^{m} \exp(\rmi\skal{\bftheta_{1k}}{(t_{k+1}-t_k)\bfd})$, $\bftheta_{11},\dots, \bftheta_{1m}\in\RR^n$, shows that $(\bfT,\bfX\circ\bfT)$ also has stationary and independent increments. Further, $(\bfT,\bfX\circ\bfT)(0)=\bfnull$ a.s.\ and the sample paths of $(\bfT,\bfX\circ\bfT)$ are a.s.\ c\`adl\`ag. Thus,  $(\bfT,\bfX\circ\bfT)\sim L^{2n}$, so we just need to verify $\Psi_{(\bfT,\bfX\circ \bfT)}=\Psi_{(\bfT,\bfX\odot \bfT)}$. 
    
    Let $\bftheta=(\bftheta_1,\bftheta_2)$, $\bftheta_1,\bftheta_2\in\RR^n$. Noting that $\bfT=\bfd I$ and using \eqref{xvecteqn}, we have
    \begin{align*}
        \Psi_{(\bfT,\bfX\circ \bfT)}(\bftheta)=\rmi \skal{\bftheta_1}{\bfd}+(\bfd\tr\Psi_{\bfX})(\bftheta_2),
    \end{align*}
    which is the same as $\Psi_{(\bfT,\bfX\odot \bfT)}(\bftheta)$ from \eqref{propsupervisexpo}.
\end{proof}

We now consider the case where $\bfT$ is a pure-jump subordinator. The following lemma establishes the Laplace functional corresponding to the PRM of the weakly subordinated process $(\bfT,\bfX\odot\bfT)$ without the finite activity assumption on $\bfT$. It is based on the marked Poisson point process of jumps of $(\bfT,\bfX\odot\bfT)$ given in the proof of \cite[Theorem 2.1 (ii)]{BLM17a}, and closely follows the arguments in the proof of \cite[Chapter VI, Theorem 3.2]{Cin11}.

Throughout the rest of this section, we let $E:=[0,\infty)\times[0,\infty)^n_*\times \RR^n$, and $Q$ be the mapping $(\bft^*,B)\mapsto \PP(\bfX(\bft)\in B)$ for $\bft^*:=(t,\bft)\in [0,\infty)\times [0,\infty)^n_*$ and Borel sets $B\subseteq\RR^n$. 

\begin{lemma}\label{prmws}
    Assume $\bfT\sim S^n(\bfnull,\TTT)$, $\TTT\neq 0$, $\bfX\sim L^n$. The PRM of the L\'evy process $(\bfT,\bfX\odot\bfT)$, denoted $\ZZ_{\odot}$, has Laplace functional 
    \begin{align}
        \EE\left[\exp \left(-\int_{E} f(\bft^*,\bfy)\,\ZZ_{\odot}(\rmd \bft^*, \rmd \bfy) \right)\right]= \EE \left[\prod_{i=1}^\infty  e^{-g(\bfT^*_i)} \right],\label{eq1}
    \end{align}
    where $f$ is a nonnegative, measurable real function,
    \begin{align*}
        e^{-g(\bft^*)}={}&\int_{\RR^n}  e^{-f(\bft^*,\bfy)} \,Q(\bft^*,\rmd\bfy),\\
        (\bfT^*_i)_{i\in\NN} :={}& (t,\Delta \bfT (t))_{t>0,\Delta \bfT (t)\neq\bfnull}.
    \end{align*}
\end{lemma}
\begin{proof}
    Noting that the L\'evy process $(\bfT,\bfX\odot\bfT)$ has only countably many jumps by Lemma \ref{countjumplem}, the Poisson point process of jumps can be written as $(t,\Delta \bfT (t),\allowbreak \Delta(\bfX\odot \bfT)(t))_{t>0,\Delta \bfT (t)\neq\bfnull} =: (\bfT^*_i,\bfY_i)_{i\in\NN}$. Here, $\bfT^*_i$ and $\bfY_i$ are random vectors that take values in $[0,\infty)\times [0,\infty)^n_*$ and $\RR^n$, respectively. So the PRM of $(\bfT,\bfX\odot\bfT)$ is
    \begin{align*}
        \ZZ_{\odot}(\rmd \bft^*, \rmd\bfy)=\sum_{i=1}^\infty \bfdelta_{(\bfT^*_i,\bfY_i)}(\rmd\bft^*, \rmd\bfy).
    \end{align*}
    Now $Q$ is a probability kernel, and $\ZZ_{\odot}$ corresponds to a marked Poisson point process with marks $(\bfY_i)_{i\in\NN}$, and $\bfY_i$, $i\in\NN$, are conditionally independent given $\bfT^*:=(\bfT^*_i)_{i\in\NN}$ with probability distribution $Q(\bfT^*_i,\rmd \bfy)$ (see  the proof of \cite[Theorem 2.1 (ii)]{BLM17a}). Consequently, the Laplace functional of $\ZZ_{\odot}$ is
    \begin{align*}
        \EE\left[\exp \left(-\int_{E} f(\bft^*,\bfy)\,\ZZ_{\odot}(\rmd \bft^*, \rmd \bfy) \right)\right] &= \EE \left[ \EE\left[\prod_{i=1}^\infty e^{-f(\bfT^*_i,\bfY_i)} \givenm \bfT^* \right]\right] \nonumber\\
        &= \EE \left[\prod_{i=1}^\infty \EE\left[ e^{-f(\bfT^*_i,\bfY_i)}\givenm \bfT^* \right]\right]\nonumber \\
        &= \EE \left[\prod_{i=1}^\infty\int_{\RR^n}  e^{-f(\bfT^*_i,\bfy)} \,Q(\bfT^*_i,\rmd\bfy) \right] \nonumber\\
        &= \EE \left[\prod_{i=1}^\infty  e^{-g(\bfT^*_i)} \right],
    \end{align*}
    as required.
\end{proof}


\begin{theorem} \label{orderjumpcaseprop} 
    Let  $\bfT\sim S^n(\bfnull,\TTT)$ and $\bfX \sim L^{n}$ be independent, with $\TTT\neq 0$ and $\TTT([0,\infty)^n_*)<\infty$. If $(\bfT,\bfX\circ\bfT)\sim L^{2n}$, then $(\bfT,\bfX\circ\bfT)\eqd (\bfT,\bfX\odot\bfT)$.
\end{theorem}

\begin{proof}

    Let $(t,\Delta \bfT (t),\Delta(\bfX\circ \bfT)(t))_{t>0,\Delta (\bfT,\bfX\circ \bfT) (t)\neq\bfnull} =: (\bfT^*_i,\bfY^*_i)_{i\in\NN}$ with $\bfT_i^*=(S_i,\allowbreak\Delta \bfT(S_i))$ and $\bfS^*:=(S_i)_{i\in\NN}$. The PRM of $(\bfT,\bfX\circ\bfT)$ is
    \begin{align*}
        \ZZ_{\circ}(\rmd \bft^*, \rmd\bfy)=\sum_{i=1}^\infty \bfdelta_{(\bfT^*_i,\bfY^*_i)}(\rmd\bft^*, \rmd\bfy).
    \end{align*}
    
    Since $\TTT([0,\infty)^n_*)<\infty$ by assumption, the jumps of $\bfT$ are countable in increasing order by Lemma \ref{countjumplem} (i). Therefore, the sample paths $t\mapsto(\bfT,\bfX\circ\bfT)(t)$ are piecewise constant a.s., which implies by Lemma \ref{lem2} (ii), that the L\'evy process $(\bfT,\allowbreak\bfX\circ\bfT)\sim L^{2n}(\bfm,\Theta,\ZZZ)$ must have $\bfm = \int_{\DD_*} (\bft,\bfx)\,\ZZZ(\rmd\bft,\rmd\bfx)$ and $\Theta = 0$.
    These are the same $\bfm$ and $\Theta$ as $(\bfT,\bfX\odot\bfT)$ in \eqref{chartri1}--\eqref{chartri2} provided that $(\bfT,\bfX\circ\bfT)$ and  $(\bfT,\bfX\odot\bfT)$ have the same L\'evy measure, which we now verify by showing they have the same PRM.
    
    Let $S_0:=0$. Recalling that $(\bfT,\bfX\circ\bfT)$ has c\`adl\`ag and piecewise constant sample paths a.s., and that $(S_i)_{i\in\NN}$ is countable in increasing order, we have
    \begin{align}
        (\Delta \bfT(S_i),\bfY_i^*)&=(\bfT(S_i)-\bfT(S_i-), \bfX\circ \bfT(S_i)-\bfX\circ \bfT(S_i-))\nonumber\\
        &=(\bfT,\bfX\circ \bfT)(S_i)-(\bfT,\bfX\circ \bfT)(S_{i-1}),\quad i\in\NN \label{eq:incr}.
    \end{align}

    
    %
    
    For $i\in\NN$, let $\PP((\Delta \bfT(S_i),\bfY_i^*)\in(\rmd\bft,\rmd\bfy) \given \bfS^*)$ denote the conditional distribution of $(\Delta \bfT(S_i),\bfY_i^*)$ given $\bfS^*$. Since $(\Delta \bfT(S_i),\bfY_i^*)$ takes values on the Borel space $([0,\infty)^n_*\times\RR^n,\BBB([0,\infty)^n_*\times\RR^n))$, there exists a regular version of the conditional probability (see \cite[Theorem 5.3]{Kal97}), so we can assume that $\PP((\Delta \bfT(S_i),\bfY_i^*)\in(\rmd\bft,\rmd\bfy) \given \bfS^*)$ is a probability kernel. 
    
    Consequently, $\PP((\Delta \bfT(S_i),\bfY_i^*)\in(\rmd\bft,\rmd\bfy) \given \bfS^*)$ is a probability measure for each value of $\bfS^*$, so it is determined by its characteristic function, which by the disintegration theorem (see \cite[Theorem 5.4]{Kal97}) is
    \begin{align}
        \hspace{2em}&\hspace{-2em} \EE\left[ \exp(\rmi\skal{ (\Delta \bfT(S_i),\bfY_i^*)  }{\bftheta}) \givenm \bfS^*\right] \nonumber\\
        ={} & \EE\left[ \exp(\rmi\skal{ (\bfT,\bfX\circ \bfT)(S_i-S_{i-1})}{\bftheta})\givenm \bfS^*\right] \nonumber\\
        ={} & \EE\left[  \exp(\rmi\skal{ \bfT(S_i-S_{i-1})}{\bftheta_1})   \exp(\bfT(S_i-S_{i-1}) \tr \Psi_{\bfX}(\bftheta_2) )   \givenm \bfS^*\right]\nonumber \\
        ={} &  \EE\left[  \exp(\rmi\skal{ \Delta \bfT(S_i) }{\bftheta_1})   \exp(\Delta \bfT(S_i) \tr \Psi_{\bfX}(\bftheta_2) )   \givenm \bfS^*\right],\label{cfxjumpt}
    \end{align}
    using the stationary increment property of $(\bfT,\bfX\circ\bfT)$, \eqref{xvecteqn}, and the stationary increment property of $\bfT$, where $\bftheta=(\bftheta_1,\bftheta_2)$, $\bftheta_1, \bftheta_2\in\RR^n$. Similarly, the conditional distribution $\PP((\Delta \bfT(S_i), \bfX(\Delta \bfT(S_i)))\in(\rmd\bft,\rmd\bfy) \given \bfS^*)$ is also a probability kernel, and using the tower law, its characteristic function is
    \begin{align*}
        \hspace{2em}&\hspace{-2em} \EE\left[ \exp(\rmi\skal{ (\Delta \bfT(S_i),\bfX(\Delta \bfT(S_i)))  }{\bftheta}) \givenm \bfS^*\right] \\
        ={} & \EE\left[  \EE\left[ \exp(\rmi\skal{ (\Delta \bfT(S_i),\bfX(\Delta \bfT(S_i)))  }{\bftheta}) \givenm \Delta \bfT(S_i), \bfS^*\right]\given \bfS^*\right], 
    \end{align*}
    which matches \eqref{cfxjumpt}.  Thus, we have
    %
    \begin{align}
        \hspace{2em}&\hspace{-2em}\PP( (\Delta \bfT(S_i),\bfY_i^*)\in (\rmd \bft, \rmd \bfy) \given  \bfS^*) \nonumber\\
        ={} & \PP((\Delta \bfT(S_i),\bfX(\Delta \bfT(S_i))) \in (\rmd \bft, \rmd \bfy)\given \bfS^*)\nonumber\\
        ={} & \PP((\Delta \bfT(S_i),\bfX(\Delta \bfT(S_i))) \in (\rmd \bft, \rmd \bfy))\nonumber\\
        ={} & \PP(\bfX(\Delta \bfT(S_i)) \in\rmd \bfy\given \Delta \bfT(S_i) = \bft)\PP(\Delta \bfT(S_i) \in \rmd \bft)\nonumber\\
        ={} & \PP(\bfX(\bft)\in\rmd \bfy)\PP(\Delta \bfT(S_i) \in \rmd \bft) \nonumber\\
        ={} & Q((t,\bft), \rmd \bfy)\PP(\Delta \bfT(S_i) \in \rmd \bft),\quad i\in\NN,\, t>0, \label{condprobmeas}
    \end{align}
    %
    where the third line is obtained by noting that the conditioning on $\bfS^*$ can be dropped because, by applying Lemma \ref{lem2} (iii) to $(\bfT,\bfX\circ \bfT)$, we see that $\bfS^*$ are the times of the jumps of a compound Poisson process, which are  independent of the size of the jumps with distribution $(\Delta \bfT(S_i),\bfX(\Delta \bfT(S_i)))$, the fourth line is obtained by \cite[Chapter 5, Equation (7)]{Kal97}, and the fifth line is obtained by the independence of $\bfT$ and $\bfX$. 
    
    %

    By definition, $\bfS^*$ are the jumping times of $(\bfT,\bfX\circ \bfT)$. Since $\bfT$ is a pure-jump subordinator with finite activity, by examining the sample paths, $\bfX\circ\bfT$ cannot jump unless $\bfT$ does. So almost surely, $ \{t>0: \Delta(\bfX\circ\bfT)(t)\neq \bfnull \}\subseteq \{t>0: \Delta\bfT(t)\neq \bfnull \}$, which implies $\bfS^*=\{t>0: \Delta\bfT(t)\neq \bfnull \}$. Thus, $\bfS^*$ are also the jumping times of $\bfT$.

    Next, for any nonnegative, measurable real function $f$, we have
    \begin{align}
        \EE\left[ e^{-f(S_i,\Delta \bfT(S_i),\bfY^*_i)} \givenm \bfS^* \right] ={} &  \int_{[0,\infty)^n_*\times\RR^n} e^{-f(S_i,\bft,\bfy)} \,Q((t,\bft),\rmd \bfy) \PP(\Delta \bfT(S_i) \in \rmd \bft)\nonumber\\
        ={} & \int_{[0,\infty)^n_*} e^{-g(S_i,\bft)}\,\PP(\Delta \bfT(S_i) \in \rmd \bft)\nonumber\\
        ={} & \int_{[0,\infty)^n_*} e^{-g(S_i,\bft)}\,\PP(\Delta \bfT(S_i) \in \rmd \bft\given \bfS^*)\nonumber\\
        ={} &\EE \left[e^{-g(S_i,\Delta\bfT(S_i))}\givenm \bfS^* \right], \quad i\in\NN,\, t>0, \label{gcondons}
    \end{align}
    where the first line follows from  the disintegration theorem (see \cite[Theorem 5.4]{Kal97}) and using \eqref{condprobmeas}, the third line follows from a similar argument as above, $\bfT$ being a compound Poisson process means the time and size of the jumps are independent, and the final line follows from another application of the disintegration theorem.
    
    Putting this together, we  compute the Laplace functional of $\ZZ_{\circ}$,
    \begin{align*}
        \EE\left[ \exp\left(-\int_{E} f(\bft^*,\bfy)\,\ZZ_{\circ}(\rmd \bft^*, \rmd \bfy) \right)\right] ={}& \EE \left[ \EE\left[\prod_{i=1}^\infty e^{-f(S_i,\Delta \bfT(S_i),\bfY^*_i)} \givenm \bfS^* \right]\right] \\
        ={} & \EE \left[\prod_{i=1}^\infty \EE\left[ e^{-f(S_i,\Delta \bfT(S_i),\bfY^*_i)} \givenm \bfS^* \right]\right]\\
        ={} & \EE \left[\prod_{i=1}^\infty \EE\left[e^{-g(S_i,\Delta\bfT(S_i))}  \givenm \bfS^* \right]\right]\\
        ={} & \EE \left[ \EE\left[\prod_{i=1}^\infty e^{-g(S_i,\Delta   \bfT(S_i))}  \givenm \bfS^* \right]\right]\\
        ={} & \EE\left[\prod_{i=1}^\infty e^{-g(\bfT_i^*)}  \right]
    \end{align*}
    where the second line follows since  $(\Delta \bfT(S_i),\bfY_i^*)$, $i\in\NN$, are conditionally independent given $\bfS^*$ because of \eqref{eq:incr} and the independent increment property of the L\'evy process $(\bfT,\bfX\circ \bfT)$, the third line follows from \eqref{gcondons}, and the fourth line  follows from a similar argument as the second line but applied to $\bfT$.  Thus, we have proven  $\ZZ_{\circ}=\ZZ_{\odot}$ from \eqref{eq1}, and hence  $(\bfT,\bfX\circ\bfT)$ and $(\bfT,\bfX\odot\bfT)$ have the same L\'evy measure, and hence the same characteristic triplet.
\end{proof}
%
%

%
%

\begin{remark}
    It is difficult to extend Theorem \ref{orderjumpcaseprop} to the more general case of a finite activity subordinator with nonzero drift or to all pure-jump subordinators. In the former case, the proof's reliance on the properties of the compound Poisson process would fail. In the latter case, for any pure-jump subordinator $\bfT$, we can create a finite activity subordinator $\bfT^{(k)}$ by truncating the size of the jumps to $\{\|\bft\|\in(1/k,\infty)\}$, $k>0$, but it is not clear the strongly subordinated process $(\bfT^{(k)},\bfX\circ\bfT^{(k)})$ is a L\'evy process to which Theorem \ref{orderjumpcaseprop} can be applied.
\end{remark}

\begin{remark}
    It would be ideal if the assumption $(\bfT,\bfX\circ \bfT)\sim L^{2n}$ in Theorem~\ref{orderjumpcaseprop} could be replaced with the weaker assumption $\bfX\circ \bfT\sim L^{n}$.  In the proof of Proposition~\ref{detsubordcase}, it is shown that $\bfX\circ \bfT\sim L^{n}$ implies $(\bfT,\bfX\circ \bfT)\sim L^{2n}$ under the assumption that $\bfT$ is deterministic. We conjecture this result  holds in general, although it is not clear how this can be proven.  In Theorem~\ref{orderjumpcaseprop}, the assumption $(\bfT,\bfX\circ \bfT)\sim L^{2n}$ on the joint process is crucial to the proof.
\end{remark}


\section{Discussion} \label{sect4}
Let  $\bfT\sim S^n$ and  $\bfX\sim L^n$ be independent with $n\geq 2$. There are some simply stated but open questions on subordination of L\'evy processes: 
\begin{itemize}
    \item If $\bfX \circ \bfT$ is a L\'evy process, then necessarily $\bfX\circ \bfT\eqd \bfX\odot \bfT$?
    \item What are the necessary and sufficient conditions on $\bfT$ and $\bfX$ such that $\bfX\circ\bfT$ is a L\'evy process?
    \item If $\bfX\circ\bfT$ is not a L\'evy process, what are the necessary and sufficient conditions on $\bfT$ and $\bfX$ such that it can be \emph{mimicked} by some L\'evy process $\bfY$ in the sense that  $(\bfX\circ\bfT)(t) \eqd\bfY(t) $ for all $t\geq 0$?
\end{itemize}

On the first question, if the answer is yes, then there cannot exist a different way to define the law of weak subordination for the class of subordinators $\bfT\sim S^n$ and  subordinates $\bfX\sim L^n$ such that $\bfX\circ \bfT\sim L^n$. Otherwise, if the answer is no, it would be interesting to determine the characteristics of $\bfX\circ \bfT$. Here, we have shown that the answer is yes if $\bfT$ is a deterministic subordinator, and it also holds under the stronger assumption $(\bfT,\allowbreak\bfX \circ \bfT)\sim L^{2n}$ if $\bfT$ is a pure-jump subordinator with finite activity.

On the second question, the sufficient conditions \ref{case1}--\ref{case3} are well-known. A partial converse has been given in \cite[Proposition 3.9]{BLM17a}. However, there are no known examples outside of condition \ref{case3}, where $\bfX\circ \bfT\sim L^n$ (besides the trivial cases where some components of $\bfT$ are the zero process). If necessary and sufficient conditions were known, it may indeed turn out that there are no additional L\'evy processes to which Proposition~\ref{detsubordcase} and Theorem~\ref{orderjumpcaseprop} are applicable besides those satisfying condition \ref{case3}.

We do not deal with the third question here, however, \cite[Proposition 3.4]{BLM17a} shows that a sufficient condition for  $(\bfX\circ\bfT)(t) \eqd(\bfX\odot\bfT)(t)$, for all $t\geq0$, to hold is that $\bfT=(T_1,\dots, T_n)$ has monotonic components, meaning there exists a permutation $\langle{(1),\dots,(n)}\rangle$ such that $T_{(1)}\leq \dots \leq T_{(n)}$. A partial converse is given in \cite[Proposition 3.10]{BLM17a}, which suggests that outside of the monotonic assumption, there may be no L\'evy process mimicking $\bfX\circ \bfT$. This raises the conjecture that if $\bfX\circ \bfT$ is not itself a L\'evy process, then it can be mimicked by a L\'evy process if and only if $\bfT$ has monotonic components.  Furthermore, these results also suggest that the mimicking L\'evy process, if it exists, may be $\bfX\odot\bfT$.

\subsection*{Acknowledgments}

This research was partially supported by ARC grant DP160104737. We thank the anonymous referees for their helpful comments.

\bibliographystyle{abbrv}
\bibliography{bibliography}
\addcontentsline{toc}{section}{\refname}

\end{document}